\documentclass[a4paper,oneside,reqno]{amsart}
\usepackage{graphicx}
\usepackage{amsmath,amssymb,amsthm}
\usepackage{xcolor}

\usepackage[colorlinks=true,
linkcolor=webgreen,
filecolor=webbrown,
citecolor=webgreen]{hyperref}

\definecolor{webgreen}{rgb}{0,.5,0}
\definecolor{webbrown}{rgb}{.6,0,0}

\usepackage{color}

\newtheorem{theorem}{Theorem}

\newtheorem{lemma}[theorem]{Lemma}
\newtheorem{proposition}[theorem]{Proposition}
\newtheorem*{resultA}{Theorem \ref{thm:A}}
\newtheorem*{resultB}{Theorem \ref{thm:B}}
\newtheorem*{resultC}{Theorem \ref{thm:C}}
\newtheorem*{resultD}{Theorem \ref{thm:D}}

\theoremstyle{definition}
\newtheorem{definition}[theorem]{Definition}

\newtheorem{conjecture}[theorem]{Conjecture}

\theoremstyle{remark}
\newtheorem{remark}[theorem]{Remark}

\newcommand{\seqnum}[1]{\href{https://oeis.org/#1}{\underline{#1}}}

\def\modd#1 #2{#1\ \mbox{\rm (mod}\ #2\mbox{\rm )}}

\allowdisplaybreaks

\begin{document}
\title[Sum-free sets generated by the period-$k$-folding sequences]{Sum-free sets generated by the period-$k$-folding sequences and some Sturmian sequences}

\author[J.-P. Allouche]{Jean-Paul Allouche}
\address[J.-P. Allouche]{CNRS, IMJ-PRG,
Sorbonne Universit\'e,
4 Place Jussieu,
F-75252 Paris Cedex 05, France}
\email{jean-paul.allouche@imj-prg.fr}

\author[J. Shallit]{Jeffrey Shallit}
\address[J. Shallit]{School of Computer Science, University of Waterloo,
Waterloo, ON  N2L 3G1, Canada}
\email{shallit@uwaterloo.ca}

\author[Z.-X. Wen]{Zhi-Xiong Wen}
\address[Z.-X. Wen]{School of Mathematics and Statistics, Huazhong University of Science and Technology, Wuhan, 430074, P.R. China.} 
\email{zhi-xiong.wen@hust.edu.cn}

\author[W. Wu]{Wen Wu$^*$}\thanks{$^*$ Wen Wu is the corresponding author.}
\address[W. Wu]{School of Mathematics, South China University of Technology, Guangzhou 510641, P.R. China.\newline
\indent Department of Mathematics and Statistics, P.O. Box 68 (Pietari Kalmin katu 5), FI-00014 University of Helsinki, Finland} 
\email[Corresponding author]{wuwen@scut.edu.cn}

\author[J.-M. Zhang]{Jie-Meng Zhang}
\address[J.-M. Zhang]{School of Science, Wuhan Institute of Technology, Wuhan 430205, P.R. China.} 
\email{zhangjiemeng@wit.edu.cn}

\date{}                                           

\begin{abstract}
First, we show that the sum-free set generated by the 
period-doubling sequence is not $\kappa$-regular for any $\kappa\geq 2$.
Next, we introduce a generalization of the period-doubling sequence,
which we call the {\it period-$k$-folding sequences}.
We show that
the sum-free sets generated by the period-$k$-folding sequences also fail
to be $\kappa$-regular for all $\kappa\geq 2$.
Finally,
we study the sum-free sets generated by Sturmian sequences that begin
with `11', and their difference sequences. 
\end{abstract}

\maketitle

\setcounter{tocdepth}{1}
\tableofcontents

\section{Introduction}
\label{one}

A set of integers $S$ is called \emph{sum-free} if $S\cap(S+S)=\emptyset$, where $S+S$ is the set $\{x+y\mid x,y\in S\}$. Equivalently, $S$ is sum-free if
the equation $x+y=z$ has no solutions $x,y,z\in S$.
When we speak of a subset $S\subset \mathbb{N}$,
we always arrange its elements in ascending order
and treat it as an integer sequence. We write $S=(s_n)_{n\geq 1}$.

One can construct an infinite sum-free set from an infinite zero-one sequence using a natural map between $\Sigma$ and $\mathfrak{S}$ introduced by 
Cameron \cite{Pj}, where $\Sigma$ and $\mathfrak{S}$ denote the set of all zero-one sequences and the set of all sum-free sets, respectively.   We explain
this map in Section~\ref{theta-sec}.
Calkin and Finch \cite{Nc} in 1996 showed that this map, denoted by $\theta$, is a bijection. One might expect to be able to characterize a sum-free set
in terms of its corresponding zero-one sequence $\mathbf{t}$. However, this is not 
always easy, even when $\mathbf{t}$ is periodic.

An infinite sum-free set $S$ is said to be (\emph{ultimately}) \emph{periodic} if its difference sequence $(s_{n+1}-s_{n})_{n\geq 1}$ is (ultimately) periodic. Calkin and Finch \cite{Nc} showed that if a sum-free set is (ultimately)
periodic, then the corresponding zero-one sequence is also (ultimately) periodic. Conversely, Cameron also asked whether sum-free sets
corresponding to (ultimately) periodic zero-one sequences are (ultimately)
periodic. This question is still open. With the help of a computer,
Calkin and Finch \cite{Nc} presented some sum-free sets, which correspond to
periodic zero-one sequences, and appear to be aperiodic (the
aperiodicity was checked up
to $10^7$).  Thus far, no proof has been provided
to show whether these particular
sum-free sets are periodic or aperiodic. Calkin and Erd\H{o}s \cite{aperiodic} showed that a class of aperiodic sum-free sets $S$ is incomplete, i.e., $\mathbb{N}\backslash (S+S)$ is an infinite set. Later, Calkin, Finch, and Flowers \cite{difference} introduced
the concept of difference density, which can be used to test whether specific sets are periodic. These tests produced further evidence that certain sets are not ultimately periodic. Payne \cite{GP}
studied the properties of certain sum-free sets over an additive group.

Wen, Zhang, and Wu \cite{WZW} studied sum-free sets
corresponding to certain zero-one automatic sequences, including
the Cantor-like sequences and some substitution sequences. Those sum-free sets were proved to be $2$-regular sequences, which implies that they have
a simple description.   In contrast, in this paper, we find that the sum-free sets corresponding to the period-doubling sequence are {\it not\/}
$\kappa$-regular for any $\kappa\geq 2$. 
\medskip

We now summarize our results. The first result characterizes the sum-free set generated by the period-doubling sequence.

\begin{theorem}\label{thm:A}
Let $(s_n)_{n\geq 1}$ be the sum-free set generated by the period-doubling sequence $\mathbf{p}$ and let $\rho_8$ be the morphism defined
by sending $0\to 833$ and $1\to 86$. Set $d_1=8$ and let $d_n=s_{n+1}-s_n$ for all $n\geq 2$. Then $(d_n)_{n\geq 1}=\rho_8(\mathbf{p})$.  Furthermore,
$(d_n)_{n\geq 1}$ is not $\kappa$-automatic for any $\kappa\geq 2$.
\end{theorem}

We introduce a general version of the period-doubling sequence.
The \emph{period-$k$-folding} sequence $\mathbf{p}^{(k)}$ is the fixed point of the morphism \[\sigma_k:\, 0\to 0^k1 \text{ and }1\to 0^{k+1},\] where $k\geq 1$ is an integer. Note that $\mathbf{p}^{(k)}$ is $(k+1)$-automatic and $\mathbf{p}^{(1)}$ is the classical period-doubling sequence \cite{Damanik:2000}. 
Define the morphism
\begin{equation*}
\tau_k: \begin{cases}1 \to 1^{k-1}2;\\
2\to 1^{k-1}21^{k+1}.\end{cases}
\end{equation*}
The sum-free set generated by the period-$k$-folding sequence is related to the morphism $\tau_{k}$ in the following way.
\begin{theorem}\label{thm:B}
Let $S=(s_n)_{n\geq 1}$ be the sum-free set generated by $\mathbf{p}^{(k)}$, 
where $k\geq 2$. Then $(s_{n+1}-s_n)_{n\geq 1}=\rho_1(\tau_k^{\infty}(1))$, where $\rho_1$ is the coding $1\to k+2$ and $2\to 2k+4$.
\end{theorem}

Next result shows the non-automaticity of the sequences $\tau_{k}^{\infty}(1)$. Therefore, by Theorem \ref{thm:B}, the sum-free set generated by period-$k$-folding sequence are not $\kappa$-regular for any $\kappa\geq 2$.

\begin{theorem}\label{thm:C}
The sequence $\tau_k^{\infty}(1)$ is not $\kappa$-automatic for any $\kappa\geq 2$.
\end{theorem}

It is also interesting to investigate the sum-free sets generated 
by some non-automatic sequences. For example, the famous class of non-automatic sequences: the Sturmian sequences. 

\begin{theorem}\label{thm:D}
The difference sequences of the sum-free sets generated by the Sturmian sequences 
beginning with `$\,11$' are also Sturmian sequences.
\end{theorem}

We focus on the Sturmian sequences beginning with `$11$' for technical reasons.
It remains unknown if similar phenomenon occurs for other Sturmian sequences.
\bigskip

The paper is organized as follows. In Section~\ref{two}, we introduce the bijection $\theta$ and give some basic facts about the sum-free  sets. In Section~\ref{three}, we
study the sum-free sets generated by period-$k$-folding sequences and prove Theorem \ref{thm:A} and \ref{thm:B}. In Section~\ref{four}, we prove Theorem \ref{thm:C}, which is the non-automaticity of the sequence $\tau_k^{\infty}(1)$.
In Section~\ref{five}, sum-free sets generated by certain Sturmian sequences are investigated and Theorem \ref{thm:D} is proved.  Finally, in Section~\ref{six}, we
give a conjecture about subword complexity.

\section{Preliminaries}
\label{two}
\subsection{Notations and definitions} For a detailed discussion about the following terms, such as ``$\kappa$-automatic sequence'', ``$\kappa$-regular sequence'', ``Sturmian sequence'', and so forth, see \cite{kreg1,kreg2,AS,M02}.
\medskip

\noindent\textbf{Words.} An \textit{alphabet} $\mathcal{A}$ is a finite set. The elements of $\mathcal{A}$ are called \textit{letters}. The set of all finite words over the alphabet $\mathcal{A}$ is $\mathcal{A}^{*}:=\cup_{\geq 0}\mathcal{A}^{n}$, where $\mathcal{A}^{0}=\{\varepsilon\}$ and $\varepsilon$ denotes the empty word. For $w\in\mathcal{A}^{*}$, let $|w|$ denote the \textit{length} of $w$. Namely, if $w\in\mathcal{A}^{n}$, then $|w|=n$. For two words $w=w_1w_{2}\cdots w_{|w|},\ v=v_{1}v_{2}\cdots v_{|v|}\in\mathcal{A}^{*}$, their \textit{concatenation} is $wv=w_1w_{2}\cdots w_{|w|}v_{1}v_{2}\cdots v_{|v|}$. 
\medskip

\noindent\textbf{Morphism.} A \textit{morphism} $\sigma$ on $\mathcal{A}$ is the map $\mathcal{A}\to \mathcal{A}^{*}$, which can be extended to $\mathcal{A}^{*}$ satisfying $\sigma(wv)=\sigma(w)\sigma(v)$ for all $w, v\in\mathcal{A}^{*}$. Let $\mathcal{A}^{\omega}$ be the set of infinite sequence on $\mathcal{A}$ and let $\mathcal{A}^{\infty}:=\mathcal{A}^{\omega}\cup \mathcal{A}^{*}$. For any $a\in\mathcal{A}$, by $\sigma^{\infty}(a)$ we mean the limit $\lim_{n\to\infty}\sigma^{n}(a)$, provided the limit exists. The limit is taken under the natural metric on $\mathcal{A}^{\infty}$.
\medskip

\noindent\textbf{$\kappa$-automatic sequences and $\kappa$-regular sequences.} Let $\kappa\geq 2$ be an integer. A sequence $\mathbf{u}=(u_{n})_{n\geq 0}$ over
the alphabet $\mathcal{A}$ is a \textit{$\kappa$-automatic} sequence if and only if its $\kappa$-kernel $K_{\kappa}(\mathbf{u})$ is finite, where $K_{\kappa}(\mathbf{u}):=\{(u_{\kappa^{i}n+j})_{n\geq 0} \mid i\geq 0,\ 0\leq j < \kappa^{i}\}$. A sequence $\mathbf{u}=(u_{n})_{n\geq 0}$ 
taking values in $\mathbb{Z}$ is \textit{$\kappa$-regular} if the $\mathbb{Z}$-module generated by its $\kappa$-kernel $K_{\kappa}(\mathbf{u})$ is finitely generated. 
\medskip

\noindent\textbf{Sturmian sequences.}  For $\mathbf{w}\in\mathcal{A}^{\infty}$, its \textit{subword complexity} function $P_{\mathbf{w}}:\mathbb{N}\to\mathbb{N}$ is defined by 
\[P_{\mathbf{w}}(n)=\#\{w_{i}w_{i+1}\cdots w_{i+n-1} \mid i\geq 0\}.\]
A sequence $\mathbf{w}$ is a \textit{Sturmian} sequence if $P_{\mathbf{w}}(n)=n+1$ for all $n\geq 1$.

\subsection{The bijection $\theta$}
\label{theta-sec}

Let $\mathbf{w}=w_1w_2w_3\cdots\in\{0,1\}^{\infty}$.
We now construct sets $S_i,T_i,U_i$, as follows.
Define $S_0=T_0=U_0=\emptyset$.
For $i=1,2,3,\ldots$, 
let $n_i$ be the least element of
$\mathbb{N}\backslash(S_{i-1}\cup T_{i-1}\cup U_{i-1})$. 
If $w_i=1$, set
\[S_i=S_{i-1}\cup\{n_i\},\quad T_i=S_i+S_i,\quad U_i=S_{i-1},\]
while if $w_i=0$, set
\[S_i=S_{i-1},\quad T_i=T_{i-1},\quad U_i=U_{i-1}\cup\{n_i\}.\]
Let $S=\bigcup_i S_i$. Then, since each $S_i$ is sum-free and $S_i\subset S_{i+1}$, the set $S$ is also sum-free. We define $S$ to be the image of $\mathbf{w}$ under $\theta$, i.e., $\theta(\mathbf{w})=S$. For example,
\begin{eqnarray*}
\theta:~11111111\cdots &\mapsto & \{1,3,5,7,9,11,13,15,\ldots\}, \\
  \theta:~01010101\cdots &\mapsto & \{2,5,8,11,\ldots\}.
\end{eqnarray*}

The inverse of $\theta$ is given as follows. Let $S\subset\mathbb{N}$ be a sum-free set with $\# S=\infty$.  We define the sequence $\mathbf{v}=(v_n)_{n\geq 1}$ over the alphabet $\{0,1,*\}$ by
\begin{equation}\label{eq:v}
v_n=\begin{cases} 1, &\text{if }n\in S;\\
*, &\text{if }n\in S+S;\\
0, & \text{otherwise.}
\end{cases}
\end{equation}
Deleting all $*$'s in $\mathbf{v}$, we obtain a zero-one sequence $\mathbf{v}^{\prime}$ and one can verify that $\theta(\mathbf{v}^{\prime})=S$.

\subsection{Basic facts}
Note that $S=\{i\in\mathbb{N}_{\geq 1}\mid v_i=1\}=(s_{n})_{n\geq 1}$ and 
\begin{equation}\label{eq:sma}
s_{n+1}-s_{n}=\mu_n+\alpha_n+1,
\end{equation}
where \[\mu_n:= \# \{i\in\mathbb{N}\mid v_i=0,\ s_{n}<i<s_{n+1}\}\]
and \[\alpha_n:=\# \{i\in\mathbb{N}\mid v_i=*,\ s_{n}<i<s_{n+1}\}.\] The quantity $\mu_n$ (resp., $\alpha_n$) is the number of `$0$'s (resp., `$*$'s) between the $n$-th and the $(n+1)$-th occurrence of `1' in $\mathbf{v}$. Moreover, $\mu_n$ also counts the number of `$0$'s between the $n$-th and the $(n+1)$-th occurrence of `1' in $\mathbf{v}^{\prime}$. Let $S^{\prime}=\{i\in\mathbb{N}\mid v^{\prime}_i=1\}=(s^{\prime}_n)_{n\geq 1}$. Then
\begin{equation}\label{eq:mu}
\mu_n=s^{\prime}_{n+1}-s^{\prime}_n.
\end{equation}

\section{Sum-free sets generated by period-\texorpdfstring{$k$}{k}-folding sequences}
\label{three}
Let $k\geq 1$ be an integer. Recall that $\sigma_k$ is the morphism  $0\to 0^k1$ and $1\to 0^{k+1}$ and the \emph{period-$k$-folding sequence} \[\mathbf{p}^{(k)}=(p_n)_{n\geq 0}=\sigma_k^{\infty}(0).\]
The sequence $\mathbf{p}^{(k)}$ can be also defined recursively by the following recurrence relations: $p_0=0$ and for all $n\geq 0$,
\begin{equation}\label{eq:gpd}
p_{(k+1)n+j}=\begin{cases} 0, & \text{ if }j=0,1,\dots, (k-1); \\
1-p_{n}, & \text{ if } j=k.\end{cases}
\end{equation}
Theorem \ref{thm:B} says that the sum-free set corresponding to $\mathbf{p}^{(k)}$ is related to the following morphism
\begin{equation*}
\tau_{k}: \begin{cases}1 \to 1^{k-1}2;\\
2\to 1^{k-1}21^{k+1}.\end{cases}
\end{equation*}

We remark that $\tau_{k}^{\infty}(1)$ is the image of the period-doubling sequence under a non-uniform projection. That is, we have
\begin{equation}\label{eq:projection}
\tau_{k}^{\infty}(1)=\rho_0(\sigma_k^{\infty}(0)),
\end{equation}
where $\rho_0$ maps $ 0\to 1^{k-1}2$ and $1\to 1^{k+1}$. One can verify 
Eq.~\eqref{eq:projection} by arguing that for all $n\geq 1$ we have
\begin{equation}\label{eq:tau-n}
\tau_{k}^{n+1}(1)=\rho_0(\sigma_k^n(0)).
\end{equation} 
Note that $\tau_{k}^2(1)=(1^{k-1}2)^{k}1^{k+1}=\rho_0(\sigma_k(0))$,
and we suppose that $\tau_{k}^{m+1}(1)=\rho_0(\sigma_k^m(0))$ for all $m\leq n$. Then
\begin{align*}
\tau_{k}^{n+2}(1) & = \tau_{k}^{n+1}(1^{k-1}2)\\
& = (\tau_{k}^{n+1}(1))^{k-1}\tau_{k}^{n+1}(2)\\
& =[\tau_{k}^{n+1}(1)]^{k}[\tau_{k}^{n}(1)]^{k+1}\\
& =[\rho_0(\sigma_k^n(0))]^{k}[\rho_0(\sigma_k^{n-1}(0))]^{k+1}\\
& = \rho_0\Big(\sigma_k^{n}(0^k)\sigma_k^{n-1}(0^{k+1})\Big) \\
& = \rho_0(\sigma_k^{n+1}(0)).
\end{align*}
So Eq.~\eqref{eq:tau-n} follows by induction.

\subsection{The blocks of zeros}
Let $\Gamma$ be the map between $\{0,1\}^{*}$ and $\mathbb{N}^{*}$ that
measures the distance between adjacent `$1$'s in finite binary
words. More precisely, if $w=0^{x_0}10^{x_1}1\cdots 0^{x_n}10^{x_{n+1}}$, 
then \[\Gamma(w)=x_1x_2\cdots x_n,\]
where $x_i\in\mathbb{N}$ for $i=0,1,\ldots, n+1$. For $w,v\in\{0,1\}^*$ and $x\geq 0$, we have
\begin{equation}\label{eq:gamma}
\Gamma(w0^x1v)=\Gamma(w0^x1)\Gamma(0^x1v).
\end{equation}

Let $\rho_2$ be the coding $1\to k$ and $2\to 2k+1$.

\begin{lemma}\label{lem:gamma}
For all $k\geq 1$, we have $\Gamma\big([\sigma_{k}^n(0)]^j\sigma_{k}(0)\big)=[\rho_2(\tau_{k}^{n-1}(k))]^j$ for all $n,j\geq 1$.
\end{lemma}
\begin{proof}
Note that $\sigma_{k}(0)=0^k1$ is a prefix of $\sigma_{k}^{n}(0)$ for all $n\geq 1$.
By Eq.~\eqref{eq:gamma} we have
\begin{equation}\label{eq:g1}
\Gamma([\sigma_{k}^n(0)]^j\sigma_{k}(0))=[\Gamma(\sigma_{k}^n(0)\sigma_{k}(0))]^j
\end{equation}
 for all $j\geq 1$. So we only need to show that for all $n\geq 1$, we have
\begin{equation}\label{eq:g2}
\Gamma(\sigma_{k}^n(0)\sigma_{k}(0))=\rho_2(\tau_{k}^{n-1}(1)).
\end{equation}
Obviously, Eq.~\eqref{eq:g2} holds for $n=1$ and $2$. Suppose that Eq.~\eqref{eq:g2} holds for all $m\leq n$. We have
\begin{align*}
\Gamma(\sigma_{k}^{n+1}(0)\sigma_{k}(0))
& = \Gamma(\sigma_{k}^n(0^k1)\sigma_{k}(0))\\
& = \Gamma([\sigma_{k}^n(0)]^k[\sigma_{k}^{n-1}(0)]^{k+1}\sigma_{k}(0)) \\
& =  \Gamma([\sigma_{k}^n(0)]^k\sigma_{k}(0))\Gamma([\sigma_{k}^{n-1}(0)]^{k+1}\sigma_{k}(0)) & (\text{by Eq.~}\eqref{eq:gamma})\\
& = [\Gamma(\sigma_{k}^n(0)\sigma_{k}(0))]^k[\Gamma(\sigma_{k}^{n-1}(0)\sigma_{k}(0)]^{k+1} & (\text{by Eq.~}\eqref{eq:g1})\\
& = \rho_2\Big([\tau_{k}^{n-1}(1)]^k[\tau_{k}^{n-2}(1)]^{k+1}\Big)  & (\text{by Eq.~}\eqref{eq:g2})\\
& = \rho_2(\tau_{k}^n(1)).
\end{align*}
Thus Eq.~\eqref{eq:g2} holds for $n+1$, which completes the proof.
\end{proof}

\begin{lemma}\label{lem:mu}
For all $k\geq 1$, $(\mu_n)_{n\geq 1}=\rho_2(\tau_{k}^{\infty}(1))$.
\end{lemma}

\begin{proof}
Recall that $\mu_n$ is the number of `$0$'s between the $n$-th and the $(n+1)$-th occurrence of 
`$1$' in $\mathbf{p}^{(k)}$. Note also that $\sigma_{k}^n(0)\sigma_{k}(0)$ is a prefix of 
$\mathbf{p}^{(k)}$ for all $n\geq 1$. Thus
\begin{equation}\label{eq:mulimit}
(\mu_n)_{n\geq 1}=\lim_{i\to\infty}\Gamma\big(\sigma_{k}^i(0)\sigma_{k}(0)\big).
\end{equation}
The result follows from Eq.~\eqref{eq:mulimit} and Lemma \ref{lem:gamma}.
\end{proof}

\subsection{The gaps for stars when $k\geq 2$}
While we construct the sum-free set $S$ corresponding to $\mathbf{p}^{(k)}$, we actually insert stars into $\mathbf{p}^{(k)}$ and finally obtain the ternary sequence $(v_n)_{n\geq 1}$ satisfying Eq.~\eqref{eq:v}.
\begin{lemma}\label{lem:alpha}
For all $k\geq 2$, $(\alpha_n)_{n\geq 1}=\tau_{k}^{\infty}(1).$
\end{lemma}

\begin{proof}
We prove that for all $n\geq 1$,
\begin{equation}\label{eq:a1}
\left\{
\begin{aligned}
\alpha_n &=\rho_2^{-1}(\mu_n);\\
\forall x &\in S_n+S_n,\ x\equiv \modd{k} {k+2},
\end{aligned}
\right.
\end{equation}
where $S_n=\{s_1,\ldots,s_n\}$.

Since $0^k10^k1$ is a prefix of $\mathbf{p}^{(k)}$ when $k\geq 2$, we have $s_1=k+1$, $s_2=2k+3$ and $\alpha_1=1=\rho_2^{-1}(\mu_1)$. So Eq.~\eqref{eq:a1} holds for $n=1$. Assume that Eq.~\eqref{eq:a1} holds for all $m\leq n$. By the inductive assumption, $v_{s_n+s_1}=*$ and $v_{s_n+s_1\pm i}\neq \ast$ for $i=1,\ldots,k$. By Lemma \ref{lem:mu}, we know that $\mu_j\in\{k,2k+1\}$ for all $j\geq 1$.

\textit{Case I.} If $\mu_{n+1}=k$, then $v_{s_n+i}=0$ for $i=1,\ldots,k$ and $v_{s_n+k+2}=1$. So $\alpha_{n+1}=1=\rho_2^{-1}(\mu_{n+1})$ and $s_{n+1}=s_{n}+k+2$. Therefore, in this case, Eq.~\eqref{eq:a1} holds for $n+1$.

\textit{Case II.} If $\mu_{n+1}=2k+1$, by the inductive assumption, $v_{s_n+s_1}=*$, $v_{s_n+s_2}=*$ and $v_{s_n+s_1\pm i}=0$ for $i=1,\ldots,k$. So $v_{s_n+2k+4}=1$ and $s_{n+1}=s_{n}+2k+4$. It follows that $\alpha_{n+1}=2=\rho_2^{-1}(\mu_{n+1})$ and $S_{n+1}=S_n\cup\{s_n+2k+4\}$, which implies that Eq.~\eqref{eq:a1} holds for $n+1$ in this case. 

By induction, we see Eq.~\eqref{eq:a1} holds for all $n\geq 1$ and this completes the proof.
\end{proof}

\begin{remark}
The stars occur periodically in $(v_n)_{n\geq 1}$. 
Actually, $v_n= *$ if and only if $n\equiv \modd{k} {k+2}$.
\end{remark}

\subsection{The gaps for stars when $k=1$}
In this case, we will show that the stars occur periodically in $(v_n)_{n\geq 1}$ with period $6$. The initial values of $(v_n)_{n\geq 1}$ are
\begin{center}
\begin{tabular}{c|*{13}{c}}
$n$ & 1 & 2 & 3 & 4 & 5 & 6 & 7 & 8 & 9 & 10 & 11 & 12 & 13\\
\hline
$v_n$ & 0 & 1 & 0 & $*$ & 0 & 0 & 1 & 0 & $*$ & 1 & 0 & $*$ & 1\\
\end{tabular}
\end{center}

\begin{lemma}\label{lem:vk=1}
Set $S_n:=\{i\mid v_i=1, 1\leq i\leq 14n+13\}$ and $\tilde{S}_n:=\{x~(\mathrm{mod}~14)\mid x\in S_n+S_n, x> 13\}$ for $n\geq 0$. Then for all $n\geq 1$
we have
\begin{itemize}
\item[(a)] for $0\leq j\leq 13$, $v_{14n+j}=*$ if and only if $j\in\mathcal{I}_*:=\{0,1,3,6,9,12\}$;
\item[(b)] $v_{14n+4}=0$, $\{14n+7,14n+13\}\subset S_n$ and $\tilde{S}_n= \mathcal{I}_* $.
\end{itemize}
\end{lemma}

\begin{proof}
From the initial values $v_{1},\dots,v_{13}$, we know that $S_0:=\{2,7,10,13\}\subset S$. Since \[S_0+S_0=\{4,9,12,14,15,17,20,23,26\},\] we see that (a) holds for $n=1$.
Recall that by Lemma \ref{lem:mu}, when $k=1$,
the sequence $(\mu_i)_{i\geq 1}$ is the fixed point of the morphism sending
$3\to 311$ and $1\to 3$. Therefore, $\mu_4=\mu_5=3$ yields that $S_1=S_0\cup\{21,27\}$, which implies that (b) holds for $n=1$.

Now assume that (a) and (b) hold for $m\leq n$. We shall show the validity of them for $n+1$. Using the inductive hypothesis (b) for $n$, $\tilde{S}_n= \mathcal{I}_* $ and \[S_0+S_n\supset S_0+\{14n+7,14n+13\}=\{14n+j\mid  j=9,14,15,17,20,23,26\},\] which imply that $v_{14(n+1)+j}=*$ if and only if $j\in\mathcal{I}_*$. So (a) holds for $n+1$.

Applying (a) for $n$ and $n+1$, we have the following table
\[
\begin{tabular}{c|*{14}c}
$j$ & 0 & 1 & 2 & 3 & 4 & 5 & 6 & 7 & 8 & 9 & 10 & 11 & 12 & 13\\
\hline
$v_{14n+j}$ & $*$ & $*$ &  &  $*$ &   &  & $*$ & 1 &  & $*$ & & & $*$ & 1 \\
$v_{14(n+1)+j}$ & $*$ & $*$ &  &  $*$ &  &  & $*$ &  &  & $*$ & & & $*$ &
\end{tabular}\]
Suppose $v_{14n+7}$ is the $\ell$-th `$1$' in $\mathbf{p}^{(1)}$. There are two cases $\mu_{\ell}=1$ and $\mu_{\ell}=3$.
\begin{itemize}
\item If $\mu_{\ell}=1$, then $\mu_{\ell+1}$ must be $1$ since $v_{14n+13}=1$. Further, $\mu_{\ell+2}$ must be $3$ since `$111$' is not a factor of $(\mu_{n})_{n\geq 1}$, which implies that $v_{14(n+1)+7}=1$ and $v_{14(n+1)+4}=0$. If $\mu_{\ell+3}=3$, then $v_{14(n+1)+13}=1$. If $\mu_{\ell+3}=1$, then $\mu_{\ell+4}=1$. We also have $v_{14(n+1)+13}=1$.
\item If $\mu_{\ell}=3$, then either $\mu_{\ell+1}=3$ or $\mu_{\ell+1}=\mu_{\ell+2}=1$. In both cases, we have $v_{14(n+1)+7}=1$. When $\mu_{\ell+1}=3$, either $\mu_{\ell+2}=3$ or $\mu_{\ell+2}=\mu_{\ell+3}=1$. In both cases, $v_{14(n+1)+13}=1$. When $\mu_{\ell+1}=\mu_{\ell+2}=1$, we have $v_{14(n+1)+4}=1$ and $\mu_{\ell+3}=3$. Note that $\mu_{\ell+3}=3$ indicates $v_{14(n+1)+11}=0$. However, $v_{7}=1$ and $v_{14(n+1)+4}=1$ yields $v_{14(n+1)+11}=*$ since $14(n+1)+11=[14(n+1)+4]+7$. This contradiction implies that $\mu_{\ell+1}=\mu_{\ell+2}=1$ cannot happen.
\end{itemize}
In the above two cases, we have \[S_{n+1}\subset S_n\cup\{14(n+1)+7,14(n+1)+10,14(n+1)+13\},\] which together with the inductive hypothesis $\tilde{S}_n=\mathcal{I}_*$, gives $\tilde{S}_{n+1}=\mathcal{I}_*$. So (b) also holds for $n+1$. This completes the proof.
\end{proof}

By Lemma \ref{lem:vk=1}, we are able to characterize $(\alpha_n)_{n\geq 1}$ through $(\mu_n)_{n\geq 1}$.

\begin{lemma}\label{lem:alpha-1}
Let $\alpha^{\prime}_1=4$ and $\alpha^{\prime}_n=\alpha_n$ for all $n\geq 2$. Then
\[(\alpha^{\prime}_n)_{n\geq 1}=\rho_4(\sigma_1^{\infty}(0))\] 
where $\rho_4:0\to 411,\, 1\to 42$ and $\sigma_1: 0\to 01,\, 1\to 00$.
\end{lemma}
\begin{proof}
Note that when $k=1$, if we replace $\rho_0$ by $\tau_{1}\circ\rho_0:0\to 211,\ 1\to 22$, then
Eq.~\eqref{eq:projection} still holds. Recall that $\rho_{2}$ is the coding $1\to 1,\ 2\to 3$ when $k=1$.
Set $\rho_3:=\rho_2\circ\tau_1\circ\rho_0$ which maps $0\to 311$ and $1\to 33$. Applying Eq.~\eqref{eq:projection}, one can decompose $(\mu_n)_{n\geq 1}$ into a sequence over the alphabet $\{311,33\}$ as follows
\begin{align*}
(\mu_n)_{n\geq 1} & = \rho_{2}(\tau_{1}^{\infty}(1))\qquad\qquad\ \ \mbox{(by Lemma \ref{lem:mu})}\\
& = \rho_{2}(\tau_{1}(\tau_{1}^{\infty}(1)))\\
& = (\rho_{2}\circ\tau_{1}\circ\rho_{0})\big(\sigma_{1}^{\infty}(1)\big)\quad \mbox{(by Eq.~\eqref{eq:projection})}\\
& = \rho_3(\sigma_1^{\infty}(1)).
\end{align*}

From Lemma \ref{lem:vk=1}, we have the distribution of $(v_n)_{n\geq 1}$ as follows: for $n\geq 1$,
\[
\begin{tabular}{c|*{15}c}
$j$ & $-1$ &  0 & 1 & 2 & 3 & 4 & 5 & 6 & 7 & 8 & 9 & 10 & 11 & 12 & 13\\
\hline
$v_{14n+j}$ & 1  & $*$ & $*$ & 0 &  $*$ & 0 & 0 & $*$ & 1 & 0 & $*$ & $\diamond$ & 0 & $*$ & 1
\end{tabular}\]
where  $\diamond \in \{ 0, 1\}$. Suppose $v_{14n-1}$ is the $\ell$-th `$1$' in $\mathbf{p}^{(1)}$. Then we have $\mu_{\ell}=3$, which implies $\alpha_{\ell}=4$. If $\mu_{\ell+1}=1$, then $\mu_{\ell+2}=1$ and $\alpha_{\ell+1}=\alpha_{\ell+2}=1$. If $\mu_{\ell+1}=3$, then $\alpha_{\ell+1}=2$. Thus if we treat $(\mu_n)_{n\geq 4}$ as a sequence on $\{311,33\}$, then $(\alpha_n)_{n\geq 4}$ is a sequence on $\{411,42\}$ by projecting $311\to 411,\ 33\to 42$. This proves the  lemma.
\end{proof}

\subsection{Proof of Theorem \ref{thm:A} and \ref{thm:B}} For readers' convenience, we restate   
our Theorem \ref{thm:A} and \ref{thm:B} here.
\begin{resultA}
Let $S=(s_n)_{n\geq 1}$ be the sum-free set generated by $\mathbf{p}^{(1)}$. Set $d_1=8$ and $d_n=s_{n+1}-s_n$ for all $n\geq 2$. Then
\begin{equation}\label{eq:thm2}
\mathbf{d}:=(d_n)_{n\geq 1}=\rho_8(\sigma_1^{\infty}(0)),
\end{equation}
where $\rho_8: 0\to 833,\, 1\to 86$ and $\sigma_1: 0\to 01,\, 1\to 00$. Moreover, $\mathbf{d}$ is not $\kappa$-automatic sequence for all $\kappa\geq 2$.
\end{resultA}
\begin{proof}
The formula \eqref{eq:thm2} follows from Eq.~\eqref{eq:sma}, Lemma \ref{lem:mu} and Lemma \ref{lem:alpha-1}.

Let $\psi$ be the coding $8\to 8,\ 3\to3$, and $6\to 8$. Then $\psi(\mathbf{d})$   is the fixed point of the morphism $8\to 833$ and $3\to 8$.  Allouche, Allouche, and Shallit \cite{AAS} in 2006 showed that this sequence is not $\kappa$-automatic sequence for all $\kappa\geq 2$. So $\mathbf{d}$ is also not an automatic sequence.
\end{proof}

\begin{resultB}
Let $S=(s_n)_{n\geq 1}$ be the sum-free set generated by $\mathbf{p}^{(k)}$ 
where $k\geq 2$. Then \[(s_{n+1}-s_n)_{n\geq 1}=\rho_1(\tau_{k}^{\infty}(1)),\] where $\rho_1$ is the coding $1\to k+2$ and $2\to 2k+4$.
\end{resultB}
\begin{proof}
The result follows from Eq.~\eqref{eq:sma}, Lemma \ref{lem:mu} and Lemma \ref{lem:alpha}.
\end{proof}

\section{Non-automaticity of \texorpdfstring{$\tau_k^{\infty}(1)$}{tau\_k\^{}{infty}(1)}}
\label{four}

Here we prove that $\tau_k^{\infty}(1)$ is not an automatic sequence.
\begin{resultC}
$\tau_{k}^{\infty}(1)$ is not a $\kappa$-automatic sequence for any $\kappa\geq 2$.
\end{resultC}

In what follows $\tau_k$ is the morphism defined over the alphabet $\{1, 2\}$ by
$\tau_k(1) = 1^{k-1} 2$, $\tau_k(2) = 1^{k-1} 2 1^{k+1}$. 
The iterative fixed point
of $\tau_k$ is
$$
\tau_k^{\infty}(1) = \underbrace{1^{k-1} \ 2 \ \cdots \ 1^{k-1} \ 2}_{\substack{(k-1) \ 
\mbox{\rm\small  occurrences} \\ \mbox{\rm\small of the word $1^{k-1} 2$}}}
\ 1^{k-1} \ 2 \ 1^{k+1} \cdots
$$

Letting $\sigma_1^{\infty}(1)$ denote the iterative fixed point of the morphism $\sigma_1$ defined by
$\sigma_1(1) = 121$, $\sigma_1(2) = 12221$, it is not very difficult to prove that
$\sigma_1^{\infty}(1) = 1 \tau_1^{\infty}(1)$. It was proved in \cite{Shallit} and written down in \cite{AAS} 
that the sequence $\sigma_1^{\infty}(1)$ is not $2$-automatic. Using methods similar to those in \cite{AAS} for other 
sequences, it can be proved, using a deep result of F. Durand \cite{Durand}, that, for $\kappa \geq 2$, the 
sequence $\sigma_1^{\infty}(1)$ is not $\kappa$-automatic either: this was actually done explicitly in 
\cite{Rampersad-Stipulanti}. Hence $\tau_1^{\infty}(1)$ is not $\kappa$-automatic either, for any $\kappa \geq 2$.

\bigskip

Here we will prove, inspired by the method in \cite{Shallit, AAS}, that the iterative fixed point of $\tau_k$ is 
not $\kappa$-automatic for any $\kappa \geq 2$. First we show, thanks to Durand's theorem \cite{Durand}, that it suffices 
to prove that $\tau_k^{\infty}(1)$ is not $(k+1)$-automatic. 

\begin{lemma}\label{simplification}
If the sequence $\tau_k^{\infty}(1)$ were $\kappa$-automatic for some $\kappa \geq 2$, then it would be $(k+1)$-automatic.
\end{lemma}

\begin{proof}
The transition matrix of the morphism $\tau_k$ is the matrix
$M_k = \begin{pmatrix} 
	  k-1 & 2k \\
	   1  & 1 
       \end{pmatrix}
$
whose dominant eigenvalue is $(k+1)$. Hence the sequence $\tau_k^{\infty}(1)$ is $(k+1)$-substitutive.
Thus, if it were $\kappa$-automatic for some $\kappa \geq 2$, then it would either be ultimately periodic (hence
in particular $(k+1)$-automatic), or the integers $(k+1)$ and $\kappa$ would be multiplicatively dependent 
(see \cite[Theorem~1]{Durand}). If $(k+1)$ and $\kappa$ are multiplicatively dependent, then there exist two nonzero 
integers $a$ and $b$ such that $(k+1)^a = \kappa^b$. Thus the sequence $\tau_k^{\infty}(1)$ is $(k+1)^a$-automatic, 
hence $(k+1)$-automatic. 
\end{proof}

\bigskip

To complete the proof of the non-automaticity of $\tau_k^{\infty}(1)$, we are thus going to prove that 
$\tau_k^{\infty}(1)$ is not $(k+1)$-automatic. We begin with some lemmas.

\begin{lemma}\label{elementary}
Define the sequence of integers $(W_k(n))_{n \geq 0}$ by \[W_k(n) := \dfrac{(k+1)^n - (-1)^n}{k+2}\cdot\] 
Then we have the following properties.

\begin{itemize}

	\item[(i)] $W_k(n+1) = (k+1)W_k(n) + (-1)^n$.

	\item[(ii)] $W_k(n+2) = k W_k(n+1) + (k+1) W_k(n)$.

	\item[(iii)] $W_k(n+1) + W_k(n) = (k+1)^n$.

	\item[(iv)] $k \left(\displaystyle\sum_{1 \leq \ell \leq n} W_k(\ell) \right) = 
		     \begin{cases} 
			     W_k(n+1), \ &\mbox{\rm if $n$ is odd;} \\
			     W_k(n+1) - 1, &\mbox{\rm if $n$ is even.}
		     \end{cases}$

	\item[(v)] $W_k(n) = \displaystyle\sum_{1 \leq j \leq n} (-1)^{j+1} (k+1)^{n-j}$.

	\item[(vi)] $W_k(n) = 
		     \begin{cases} 
			 (k+1)^{n-1} - k \left(\displaystyle\sum_{1 \leq j \leq \frac{n-1}{2}} (k+1)^{n-2j-1}\right),
				 \ &\mbox{\rm if $n$ is odd;} \\
			 (k+1)^{n-1} - k \left(\displaystyle\sum_{1 \leq j \leq \frac{n-1}{2}} (k+1)^{n-2j-1} \right) - 1,
				  \ &\mbox{\rm if $n\not=0$ is even;}
		     \end{cases}$ 

	\item[(vii)] The length of $\tau_k^n(1)$ is equal to $W_k(n+1)$.

\end{itemize}
\end{lemma}

\begin{proof}
Assertions (i), (ii), (iii) and (iv) are straightforward consequences of the definition 
of $W_k(n)$. Assertion~(v) is proved by induction on $n$ using (i).
Assertion (vi) is proved by calculating the sum
$\displaystyle\sum_{1 \leq j \leq \frac{n-1}{2}} (k+1)^{n-2j-1}$.

\medskip

Finally, to prove (vii), we let $\ell_k(n)$ and $m_k(n)$ denote the lengths of the words $\tau_k^n(1)$ and 
$\tau_k^n(2)$. We clearly have from the definition of $\tau_k$ that $\ell_k(0) = m_k(0) = 1$, and, for $n \geq 0$,
\begin{align*}
		\ell_k(n+1) & = (k-1) \ell_k(n) + m_k(n)\\
		m_k(n+1) & = 2k \ell_k(n) + m_k(n).
\end{align*}
Define $\ell_k'$ and $m_k'$ by $\ell_k'(n) := W_k(n+1)$ and $m_k'(n) := W_k(n+2) - (k-1) W_k(n+1)$.
Since $\ell_k'$ and $m_k'$ have the same initial values and satisfy the same recurrence (use (ii)) 
as $\ell_k$ and $m_k$, we have $\ell_k = \ell_k'$ and $m_k = m_k'$.
\end{proof}

\begin{remark}
The sequence $(W_1(n))_{n \geq 0} = 0 \ 1 \ 1 \ 3 \ 5 \ 11 \ 21 \ 43 \ \cdots$ is the {\it Jacobsthal sequence}
(sequence \seqnum{A001045} in the {\it On-Line Encyclopedia of Integer
Sequences} (OEIS) \cite{OEIS}). The sequence $(W_2(n))_{n \geq 0} = 0 \ 1 \ 2 \ 7 \ 20 \ 61 \ \cdots$ is 
sequence \seqnum{A015518} in the OEIS. The sequences $(W_3(n))_{n \geq 0}$, $(W_4(n))_{n \geq 0}$, $\ldots$, up 
to $(W_9(n))_{n \geq 0}$ are, respectively, the sequences \seqnum{A015521},
\seqnum{A015531}, \seqnum{A015540}, \seqnum{A015552}, \seqnum{A015565},
\seqnum{A015577}, \seqnum{A015585}  in 
the OEIS. The number $W_k(n)$ counts in particular the number of walks of length $n$ between two distinct vertices 
of the complete graph $K_n$. Also see Proposition~\ref{complement} below.
\end{remark}

Now we introduce a numeration system associated with $\tau_k$ (where, as previously, $k \geq 0$).
Two propositions about this numeration sytem and its relation to the sequence $\tau_k^{\infty}(1)$
will prove useful for obtaining that $\tau_k^{\infty}(1)$ is not $(k+1)$-automatic.

\begin{definition}
Let $r$ be a positive integer. Let $x_1$, $x_2$, $\ldots$,  $x_r$ be nonnegative integers.
We let $[x_r x_{r-1} \cdots x_1]_W$ denote the integer $\sum_{1 \leq j \leq r} x_j W_k(j)$.
We say that $[x_r x_{r-1} \cdots x_1]_W$ is a \emph{valid $W$-expansion} of the integer
$\sum_{1 \leq j \leq r} x_j W_k(j)$ if all the $x_j$'s belong to $[0, k]$, with $x_r \neq 0$,
and if the word $x_r x_{r-1} \cdots x_1$ ends with an even number (possibly equal to $0$) of $k$'s.
\end{definition}

\begin{proposition}
Every nonzero integer admits a unique valid $W$-expansion.
\end{proposition}

\begin{proof}
First we show that every nonzero integer admits a valid $W$-expansion,
by proving by induction on $t$ that, 
for all $n \in [1, W_k(t))$,
$n$ admits a valid $W$-expansion $n = [x_r x_{r-1} \cdots x_1]_W$ with $r < t$,
and $[x_r x_{r-1} \cdots x_1]_W$ ends with an even number (possibly equal to zero) of $k$'s. 
This is true for $t = 2$,
since $W_k(2) = k$ and we have $n = [n]_W$ for all $n \in [1, k)$.
Suppose that the property holds for some $t$,
and let $n$ be an integer belonging to $[W_k(t), W_k(t+1))$. 
Since $W_k(t) = [10^{t - 1}]_W$ 
we can suppose that $n$ belongs to $(W_k(t), W_k(t+1))$. Using 
Assertion~(i) of Lemma~\ref{elementary},
we have $W_k(t) < n < W_k(t+1) \leq (k+1) W_k(t) + 1$.  Hence 
$W_k(t) < n \leq (k+1) W_k(t)$. Thus, if $\alpha$ is the integer such that 
$\alpha W_k(t) < n \leq (\alpha + 1) W_k(t)$,
we have $\alpha < k+1$ and $\alpha + 1 > 1$. 
Hence $1 \leq \alpha \leq k$.
Define $m := n - \alpha W_k(t)$. Then $m$ belongs to $(0, W_k(t)]$.
By the induction hypothesis for $m \neq W_k(t)$ and directly for $m = W_k(t)$, $m$ can be represented as
$[x_r x_{r-1}\cdots x_1]_W$, with $r < t$ and $[x_r x_{r-1} \cdots x_1]_W$ ends with an even number (possibly equal 
to zero) of $k$'s. Then 
$$n = m + \alpha W_k(t) 
= [\alpha \!\! \underbrace{0 \ \cdots \ 0}_{(t-1-r) \ \mbox{\rm\small terms}} \!\! x_r \ x_{r-1} \ \cdots \ x_1]_W.$$
This yields a valid $W$-expansion of $n$, except possibly if $r = t - 1$, $\alpha = k$, all the $x_j$'s are equal 
to $k$, and $t$ is odd. But then $n = k(W_k(t) + W_k(t - 1) + \cdots + W_k(1))$, which is equal to $W_k(t+1)$ 
(Assertion~(iv) of Lemma~\ref{elementary}): but we assumed $n < W_k(t+1)$.

\bigskip

To prove uniqueness of the valid $W$-expansion of every integer, it suffices to prove that the number of words 
$x_r x_{r-1} \cdots x_1$ with $x_i \in [0, k]$, $x_r \neq 0$ and
such that $x_r x_{r-1} \cdots x_1$ ends with an 
even number (possibly $0$) of $k$'s
is equal to the number of integers in the interval $[W_k(r), W_k(r+1))$, 
i.e., $W_k(r+1) - W_k(r)$.
To count this number of words,
we note that it is the difference between the number
of all words of length $r$ over the alphabet $\{0, 1, \ldots, k\}$ beginning with $\alpha \in [1,k]$ 
(i.e., $k (k+1)^{r-1}$) and the number of all such words ending with $\beta k$,
or $\beta k k k$, or $\beta k k k k k \cdots $, where $\beta$ is any letter different from $k$,
except possibly if $r$ is odd and all the letters in $x_r x_{r-1} \cdots x_1$ are equal to $k$.
We thus obtain $k(k+1)^{r-1} - (k^2(k+1)^{r-3} + k^2(k+1)^{r-5} + \cdots) - \eta_r$, 
where $\eta_r$ is equal
to $0$ if $r$ is even and to $1$ if $r$ is odd, which (see Assertion~(vi) of Lemma~\ref{elementary}) 
is equal to $W_k(r) + 1$ if $r$ is even, and to $W_k(r) - 1$ if $r$ is odd, thus to $W_k(r) + (-1)^r$.
And this last quantity is equal to $W_k(r+1) - W_k(r)$ from Assertion~(i) in Lemma~\ref{elementary}.
\end{proof}

\begin{proposition}\label{value-1-2}
Let $\tau_k^{\infty}(1) := (t_k(n))_{n \geq 0} = t_k(0) t_k(1) \cdots \in \{1, 2\}^{\mathbb N}$. Then $t_k(n) = 2$ if
and only if the valid $W$-expansion of $(n+1)$ ends with an odd number of $0$'s (or equivalently if and only if the 
valid $W$-expansion of $(n+1)$ has the form $n = \sum_{2\ell+2 \leq j \leq r} x_j W_k(j)$, for some $\ell \geq 0$,
and $x_{2\ell+2} \neq 0$).
\end{proposition}

\begin{proof} First we note that 
\begin{equation}
\tau_k^{m+2}(1) = (\tau_k^{m+1}(1))^k (\tau_k^m(1))^{k+1}, \ \mbox{\rm for all } m \geq 0.
\label{astar}
\end{equation}
For $m=0$ we have
$$\tau_k^2(1) = (1^{k-1}2)^{k-1} (1^{k-1}21^{k+1} )= (1^{k-1}2)^k 1^{k+1} =
(\tau_k^1(1))^k (\tau_k^0(1))^{k+1}.$$ 
It then suffices to apply $\tau_k^m$ to this equality. 
In other words, Eq.~\eqref{astar} means that $\tau_k^{m+1}(1)$ is the concatenation of $k$ blocks equal to $\tau_k^{m+1}(1)$ 
and of $(k+1)$ blocks equal to $\tau_k^m(1)$ (of respective lengths $W_k(m+2)$ and $W_k(m+1)$ from Assertion~(vii) 
of Lemma~\ref{elementary}):
$$
\tau_k^{m+2}(1) = \underbrace{\tau_k^{m+1}(1) \ \cdots \ 
                      \tau_k^{m+1}(1)}_{\substack{\mbox{\rm\small $k$ blocks of} \\ \mbox{\rm\small length $W_k(m+2)$}}} 
		  \
		  \underbrace{\tau_k^m(1) \ \cdots \ 
		      \tau_k^m(1)}_{\substack{\mbox{\rm\small $(k+1)$ blocks of} \\ \mbox{\rm\small length $W_k(m+1)$}}}.
$$

We first prove that if $n = W_k(m)$ for some integer $m$, then $t_k(n) = 2$ if and only if the valid 
$W$-expansion of $n+1$ ends with an odd number of $0$'s.
	
-- if $n= W_k(0) = 0$, then $t_k(0) = 1$ except for $k=1$ since $t_0(0)=2$. The valid $W$-expansion of $0+1=1$ is 
$[1]_W$ if $k \geq 2$ and $[10]_W$ if $k=1$ (since $[1]_W$ is not valid in this case);

-- if $n$ belongs to $\{W_k(1), W_k(2)\} = \{1, k\}$, we have that $n+1$ belongs to $\{2, k+1\}$. Hence either $k=1$, 
thus $n+1=2$ whose valid $W$-expansion is $[11]_W$; or $k=2$ and $n+1 \in \{2,3\}$, thus $2=[10]_W$ and $3=[11]_W$; 
or $k \geq 3$, thus $2 = [2]_W$ and $k+1=[11]_W$. So, for $n \in \{W_k(1), W_k(2)\} = \{1, k\}$, the valid 
$W$-expansion of $n+1$ ends with an odd number of $0$'s if and only if $n+1=k=2$. On the other hand $t_k(1)=2$ if 
and only if $k=2$, and $t_k(k) = 1$ for all $k \geq 1$.

-- if $n = W_k(m)$ for some $m \geq 3$, then $\tau_k^{m-1}(2)$ is followed by $1$, so that $t_k(n) = 1$, 
and $n+1 = W_k(m)+1$ ends with $1$, thus with an even number of $0$'s

\medskip

Now we prove by induction on $m \geq 1$ that, for all $n \in [0, W_k(m)]$, $t_k(n) = 2$ if and only if the valid 
$W$-expansion of $n+1$ ends with an odd number of $0$'s. Note that, from what precedes, it suffices to prove 
the claim for $n \in [0, W_k(m))$:

\medskip

-- For $m=1$, hence $W_k(1) = 1$, the set of relevant $n$ is empty.

\medskip
	
-- For $m=2$, hence $W_k(2) = k$, the set of relevant $n$ that do not already satisfy $n < W_k(1)$ is 
$\{1, 2, \ldots, k-1\}$; thus $n+1$ belongs to $\{2, \ldots, k\}$ and its valid $W$-expansion either belongs to 
$\{[10]_W\}$ if $k=2$, or it belongs to $\{[2]_W, \ldots, [k-1]_W, [10]_W\}$ if $k > 2$. In both cases there is only 
one such $W$-expansion ending with an odd number of $0$'s, namely $[10]_W = 2$, giving $n = k-1$. Since the prefix 
of length $W_k(1) = k$ of $\tau_k^{\infty}(1)$ is $1^{k-1}2$, we are done with the case $m=2$.

\medskip

-- Now suppose that our claim holds for $m+2$ (hence also for $m$ and $m+1$ since $W_k$ is increasing) 
for some $m \geq 0$. We want to prove it for $m+3$. Looking at the decomposition into blocks in Eq.~\eqref{astar}, we see that $n$ 
belongs to one of $(2k+1)$ blocks. If $n < W_k(m+3)$ belongs to one of the first $k$ blocks of length $W_k(m+2)$ that 
compose $\tau_k^{m+2}(1)$, then there exists $j \in [1, k-1]$ such that $n$ belongs to $[jW_k(m+2), (j+1)W_k(m+2)-1]$.
Thus $n = jW_k(m+2)  + \ell$, where $\ell$ belongs to $[0, W_k(m+2)-1]$. Furthermore $t_k(n) = 2$ if and only if 
$t_k(\ell) = 2$. We distinguish between the case where $\ell + 1 = W_k(m+2)$ and the case where $\ell + 1 < W_k(m+2)$.
If $\ell + 1 = W_k(m+2)$, the valid $W$-expansion of $\ell + 1$ is $[10^{m+1}]_W$ if $k > 1$ or $m > 0$, and $[11]_W$ 
if $k = 1$ and $m = 0$. The valid $W$-expansion of $n+1 = (j+1)W_k(m+2)$ then ends with the same number of $0$'s as
the valid $W$-expansion of $\ell + 1$. If $\ell + 1 < W_k(m+2)$, then by the induction hypothesis 
$\ell + 1 = [x_r x_{r-1} \cdots x_1]_W$ with $r < m+2$, and the valid $W$-expansion $[x_r x_{r-1} \cdots x_1]_W$ ends
with an odd number of $0$'s if and only if $t_k(\ell) = 2$. The valid $W$-expansion of $(n+1)$ is clearly equal to  
$[j \ \underbrace{0 \ \cdots \ 0}_{\mbox{\rm\small $m+1-r$ terms}} \ x_r \ x_{r-1} \ \cdots \ x_1]_W$:
it ends with an odd number of $0$'s if and only if this is also the case for $[x_r x_{r-1} \cdots x_1]_W$
(except possibly if $n=0$, but this case has already been dealt with).

\medskip

-- It remains to study what happens when $n$ belongs to the last $(k+1)$ blocks in Eq.~\eqref{astar}.
The proof is tedious and works in exactly the same way, so that we omit it.
\end{proof}

Now we prove one last lemma before our non-automaticity theorem.
\begin{lemma}\label{construction}
For $\ell, r, n \geq 1$
define the integers $b_k(l,n)$ and $c_k(\ell,r,r)$ by
\begin{align*}
b_k(\ell,n) & = [(1 0^{2\ell - 1})^n]_W  \\
   & = \sum_{1 \leq j \leq n} \frac{(k+1)^{2j\ell} - 1}{k+2} = \frac{(k+1)^{2\ell(n+1)} - (k+1)^{2\ell}}{(k+2)((k+1)^{2\ell}-1)} - \frac{n}{k+2}, \\
c_k(\ell,r,n) &= b_k(\ell,n) - b_k(\ell,r) - (k+1) W_k(2\ell). 
\end{align*}
Then the following properties hold.
\begin{itemize}
\item[(i)] For all $n \geq 1$ we have $b_k(\ell, n) = (k+1)^{2\ell} b_k(\ell, n-1) + n \frac{(k+1)^{2\ell} - 1}{k+2}$.

\item[(ii)] For all $\ell \geq 1$ we have $\frac{(k+1)^{2\ell} - 1}{k+2} \ | \ b_k(\ell, n)$.

\item[(iii)] For all $\ell,n \geq 1,$ we have $t_k(b_k(\ell, n) - 1) = 2$. 

\item[(iv)] For all $\ell, r \geq 1$, and $n$ sufficiently
large, we have $t_k(c(\ell,r,n) - 1) = 1$.

\item[(v)] For all $\ell \geq 1$, for all $c \geq 0$, and for all $i \in [0, (k+1)^{2c})$,
there exist infinitely many $r \geq 0$ such that $b(\ell, r) \equiv 
\modd{i} {(k+1)^{2c}}$.
\end{itemize}
\end{lemma}

\begin{proof}
Claims (i) and (ii) are clear. To prove Claim~(iii), note the valid $W$-expansion
$b(\ell,n) = [(1 \ 0^{2\ell - 1})^n]_W$ and use Proposition~\ref{value-1-2}. For Claim~(iv), we first 
write (note that some $W$-expansions below are not valid $W$-expansions)
\begin{align*}
	c_k(\ell,r,n) &= [(1 \ 0^{2\ell - 1})^n]_W - [(1 \ 0^{2\ell - 1})^r]_W - (k+1) [1 \ 0^{2\ell-1}]_W \\
	&= [(1 \ 0^{2\ell - 1})^{n-r} \ 0^{2\ell r}]_W - (k+1) [1 \ 0^{2\ell-1}]_W  \\
	&= [(1 \ 0^{2\ell - 1})^{n-r-1} \ 1 \ 0^{2\ell - 1 + 2 \ell r}]_W - (k+1) [1 \ 0^{2\ell-1}]_W \\
	&= [(1 \ 0^{2\ell - 1})^{n-r-1} \ 0 \ k^{2\ell - 1 + 2 \ell r}]_W - (k+1) [1 \ 0^{2\ell-1}]_W \ \
	\mbox{(using Lemma~\ref{elementary}~(iv))} \\
	&= [(1 \ 0^{2\ell - 1})^{n-r-1} \ 0 \ k^{2\ell - 1 + 2 \ell r}]_W - [k \ 0^{2\ell-1}]_W - [1 \ 0^{2\ell-1}]_W  \\
	&= [(1 \ \ 0^{2\ell - 1})^{n-r-1} \ 0 \ k^{2 \ell r - 1} \ 0 \ k^{2 \ell -1}]_W - [1 \ 0^{2\ell-1}]_W \\
	&= [(1 \ \ 0^{2\ell - 1})^{n-r-1} \ 0 \ k^{2 \ell r - 1} \ 0 \ k^{2 \ell -1}]_W - [0 \ k^{2\ell-1}]_W \ \
	\mbox{(using Lemma~\ref{elementary}~(iv))} \\
	&= [(1 \ \ 0^{2\ell - 1})^{n-r-1} \ 0 \ k^{2 \ell r - 1} \ 0^{2 \ell}]_W.
\end{align*}
Since this last $W$-expansion is valid, Proposition~\ref{value-1-2} yields that $t(c_k(\ell,r,n) - 1) = 1$.

\medskip

Finally to prove (v), we note that both $(k+2)$ and $(k+1)^{2 \ell + 1}$ are prime to $k+1$.
Thus $b(k,r) \equiv \modd{i} {(k+1)^{2c}}$ holds if and only if
$$
r((k+1)^{2 \ell} - 1) - (k+1)^{2 \ell (r+1)} + (k+1)^{2 \ell} \equiv 
	\modd{- i (k + 2) ((k+1)^{2\ell} - 1)} {(k+1)^{2c}}.
$$
This holds for $r$ sufficiently large and congruent to 
$\modd{-i (k+2) - (k+1)^{2 \ell}((k+1)^{2 \ell} - 1)^{-1}} {(k+1)^{2c}}$.
\end{proof}

Now we are ready for the non-automaticity theorem (Theorem \ref{thm:C}).

\begin{proof}[Proof of Theorem \ref{thm:C}]
As proved in Lemma~\ref{simplification}, it suffices to show that $\tau_k^{\infty}(1)$ is not $(k+1)$-automatic. 
Recall that this is equivalent to saying that its $(k+1)$-kernel is not finite, where the $(k+1)$-kernel of the 
sequence $\tau_k^{\infty}(1) = (t_k(n))_{n \geq 0}$ is the set of subsequences
$$
\left\{\big(t_k((k+1)^a n + b)\big)_{n \geq 0}\mid \ a \geq 0, b \in [0, (k+1)^a-1]\right\}.
$$
Since $t_k((k+1)^{2c} n - 1) = t_k((k+1)^{2c}(n-1) + (k+1)^{2c} - 1)$ for $n \geq 1$, it suffices to prove that,
for all integers $c, c'$ with $0 \leq c < c'$, the sequences $(t_k((k+1)^{2c} n - 1))_{n \geq 0}$ and 
$(t_k((k+1)^{2c'} n - 1))_{n \geq 0}$ are distinct. Let $\ell = c' - c$. From Lemma~\ref{construction}~(v)
applied to $i \equiv \modd{(k - (k+1)^{2 \ell + 1})(k + 2)^{-1}} {(k+1)^{2c}}$,
there exist infinitely many $r$ such 
that $b_k(\ell, r) \equiv \modd{(k - (k+1)^{2 \ell + 1})(k + 2)^{-1}} {(k+1)^{2c}}$. Let 
$m = k + 1 + \frac{b_k(\ell,r)(k+2)}{(k+1)^{2 \ell}-1}$. This is an integer by Lemma~\ref{construction}~(ii) and
$$
m - 1 = k + \frac{b_k(\ell,r)(k+2)}{(k+1)^{2 \ell} - 1} 
	\equiv \modd{- \frac{(k+1)^{2 \ell}}{(k+1)^{2 \ell} - 1}} {(k+1)^{2c}}.
$$
But (see Lemma~\ref{value-1-2}) the expression of 
$$b_k(\ell, m-1) = \frac{(k+1)^{2\ell m} - (k+1)^{2\ell}}{(k+2)((k+1)^{2\ell}-1)} - \frac{m-1}{k+2}$$
shows that,
provided $\ell m \geq c$ (which holds if $r$ is sufficiently large), we have
$$
b_k(\ell, m-1) \equiv 
\modd{- \frac{(k+1)^{2\ell}}{(k+2)((k+1)^{2\ell}-1)} - \frac{m-1}{k+2}}
{(k+1)^{2c}}.
$$
Thus $b_k(\ell, m-1) \equiv \modd{0} {(k+1)^{2c}}$. Let $j = b_k(\ell, m-1) / (k+1)^{2c}$. Using 
Lemma~\ref{construction}~(i) we have
\begin{align*}
	(k+1)^{2 \ell} b_k(\ell, m - 1) &= b_k(\ell, m) - \dfrac{(k+1)^{2 \ell} - 1}{k+2} m \\
&= b_k(\ell, m) - \dfrac{(k+1)^{2 \ell} - 1}{k+2} \left(k + 1 + \dfrac{b_k(\ell,r)(k+2)}{(k+1)^{2 \ell} - 1 } \right) \\
&= b_k(\ell, m) - (k+1) \dfrac{(k+1)^{2 \ell} - 1}{k+2} - b_k(\ell,r) \\
&= b_k(\ell, m) - b_k(\ell,r) - (k+1) W_k(2 \ell) \\
&= c_k(\ell,r,m).
\end{align*}
	By Lemma~\ref{construction}~(iii) and (iv), we have 
$$
t_k(b_k(m - 1) - 1) = 2 \ \ \mbox{\rm and} \ \ t_k((k+1)^{2 \ell} b_k(\ell, m - 1) - 1)
= t_k(c_k(\ell,r,m) - 1) = 1.
$$
Thus 
$$
	t_k((k+1)^{2c}j - 1) = t_k(b_k(\ell, m-1) - 1) = 2,
$$
while
\begin{align*}
t_k((k+1)^{2c'}j - 1) &= t_k((k+1)^{2c'-2c}b_k(\ell, m-1) - 1)\\
& = t_k((k+1)^{2\ell}b_k(\ell, m-1) - 1) \\
&= t_k(c_k(\ell,r,m)-1) = 1,
\end{align*}
which shows that the sequences $(t_k((k+1)^{2c} n) - 1))_{n \geq 0}$ and $(t_k((k+1)^{2c'} n) - 1))_{n \geq 0}$ 
are distinct.
\end{proof}	
	
As recalled above, the sequence $1 \ \tau_1^{\infty}(1) = 1 \ 2 \ 1 \ 1 \ 2 \ 2 \ \cdots$ is also
the iterative fixed point of a morphism, namely $1 \ \tau_1^{\infty}(1) = \sigma_1^{\infty}(1)$ where
$\sigma_1$ is defined by $\sigma_1(1) = 121$, $\sigma_1(2) = 12221$. One can ask whether a similar 
property holds for $1 \ \tau_k^{\infty}(1)$. The following proposition answers this question.

\begin{proposition}\label{complement}
For $k \geq 1$, define the morphism $\sigma_k$ by 
$$
\sigma_k(1) = 1 \ (1^{k-1} \ 2)^k \ 1^k, \ \sigma_k(2) = 1 \ (1^{k-1} \ 2)^{2k+1} \ 1^k.
$$
Then $\sigma_k$ has an iterative fixed point that satisfies
$$
\sigma_k^{\infty}(1) = 1 \ \tau_k^{\infty}(1).
$$
\end{proposition}

\begin{proof}
First we note that 
$$
	(1^{k-1} \ 2)^k	\ 1^k \ \sigma_k(1) = \tau_k^2(1) \ (1^{k-1} \ 2)^k \ 1^k \ \mbox{\rm and} \
	(1^{k-1} \ 2)^k \ 1^k \ \sigma_k(2) = \tau_k^2(2) \ (1^{k-1} \ 2)^k \ 1^k
$$
so that for words $z \in \{1, 2\}^*$ we have
$$
\sigma_k(1z) = \sigma_k(1) \sigma_k(z) = 1 \ (1^{k-1} \ 2)^k \ 1^k \ \sigma_k(z) 
= 1 \ \tau_k^2(z) \ (1^{k-1} \ 2)^k \ 1^k.
$$
Applying this equality to $z = z_k(m)$ the prefix of length $m$ of $\tau_k^{\infty}(1)$, and letting $m$ tend 
to infinity, we obtain 
$$
	\sigma_k(1 \ \tau_k^{\infty}(1)) = 1 \ \tau_k^2(\tau_k^{\infty}(1)) = 1 \ \tau_k^{\infty}(1).
$$
Hence $1 \ \tau_k^{\infty}(1)$ is a fixed point of $\sigma_k$, thus the iterative fixed point of $\sigma_k$.
\end{proof}

\subsection{Miscellanea}
\label{misc}

The sequence $\sigma_1^{\infty}(1)$ appears in several places in the literature: it was used by Brlek
\cite{Brlek} for determining the block-complexity of the Thue-Morse sequence. We have already cited 
\cite{AAS} where it is related to an Indian {\it kolam}. It also occurs in \cite{ABPS} in relation to a
piecewise affine map. Finally we want to point out that the sequence of integers $(W_k(n))_{n \geq 1}$
actually occurred (under another name) in \cite{AABBJPS}; the following result that relates the sequence 
$\tau_k^{\infty}(1)$, the sequence $(W_k(n))_{n \geq 0}$, and a ``something-free'' set, is proven there:

\begin{theorem}[\cite{AABBJPS}]
Let $\sigma_k^{\infty}(1) = (s_k(n))_{n \geq 0}$. Define $c_k(n) = \sum_{0 \leq j \leq n} s_k(n)$.
Then the sequence ${\mathbf c} = (c_k(n))_{n \geq 0}$ has the property that $n$ belongs to ${\mathbf c}$
if and only if $(k+1)n$ does not belong to ${\mathbf c}$. Furthermore it admits the following generating function
$$
\sum_{n \geq 0} c_k(n) x^n = \frac{1}{1-x} \prod_{j \geq 1} \frac{1 - x^{(k+1) W_k(j)}}{1 - x^{W_k(j)}} \cdot
$$
\end{theorem}

\section{Sum-free sets generated by certain Sturmian sequences}
\label{five}
Let $\mathbf{t}=(t_n)_{n\geq 0}$ be a sequence on $\{0,1\}$ and let $S=(s_n)_{n\geq 1}$ be the sum-free set corresponding to $\mathbf{t}$. Recall that $\mathbf{v}=(v_n)_{n\geq 1}$ is the sequence defined by Eq. \eqref{eq:v} according to $S$.

\begin{lemma}\label{star-periodic}
If $t_0=t_1=1$ and `$\,00$' does not occur in $\mathbf{t}$, then for all $n\geq1$, 
we have $v_{n}=\ast$ if and only if $n$ is even.
\end{lemma}

\begin{proof}
Note that $t_0=t_1=1$. By the construction of $(v_n)_{n\geq 1}$, we have
\[
  \begin{array}{c|cccccc}
    n & 1 & 2 & 3 & 4 & 5 & 6  \\
    \hline
    v_n & 1 & \ast & 1 & \ast &  & \ast \
  \end{array}\cdot
\]
When $t_2=1$, then $v_5=1$ and $v_8=v_{10}=\ast$. We have $s_1=1$, $s_2=3$ and $s_3=5$.
When $t_2=0$, since `00' does not occur in $(t_n)_{n\geq0}$, we have $t_3=1$. Hence $v_5=0$, $v_7=1$, and $v_8=v_{10}=\ast$. We have $s_1=1$, $s_2=3$ and $s_3=7$. So the result holds for all $n\leq 6$. Assume that the result holds for all $m < n$. We prove it for $m=n$.
\medskip

\noindent\textit{Case 1}: $n=2k$. If $v_{2k-1}=1$, then $v_{2k}=\ast$ since $v_{1}=1$. If $v_{2k-1}=0$, then $v_{2k-2}=\ast$ by the inductive hypothesis and $v_{2k-3}=1$ since `00' does not occur. Note that $v_{3}=1$, we also have $v_{2k}=\ast$.
\medskip

\noindent\textit{Case 2}: $n=2k+1$. By the inductive hypothesis, for all $\ell\leq 2k$, if $v_{\ell}=1$, then $\ell$ is odd. If $v_{2k+1}=\ast$, then there exist $i,j\leq 2k$ such that $v_{i}=v_{j}=1$ and $2k+1=i+j$. This contradicts to the fact that both $i$ and $j$ are odd. So $v_{2k+1}\neq \ast$.

Now we see the result is valid for $n$ and our lemma follows from induction.
\end{proof}
\begin{remark}
Under the assumption of Lemma \ref{star-periodic}, we see that $(s_n)_{n\geq 1}$ are odd numbers.
\end{remark}

Now we shall discuss the sum-free set $S=(s_{n})_{n\geq 1}$ generated by a Sturmian sequence 
$\mathbf{t}$ with $t_0=t_1=1$. Note that `$\,00$' cannot occur in $\mathbf{t}$ (namely, a Sturmian
sequence has $\ell + 1$ factors of length $\ell$; since it is not periodic, it always contain `$\,01$' 
and `$\,10$', thus if `$11$' is a factor, the Sturmian sequence has exactly $3$ factors of length $2$, 
i.e., `$00$', `$01$' and `$10$'). Let \[(d_n)_{n\geq 1}:=(s_{n+1}-s_{n})_{n\geq 1}\] be the difference sequence of $S$. It is interesting to see that the difference sequence is still a Sturmian sequence. We restate our 
Theorem \ref{thm:D} as follow.

\begin{resultD}
If $\mathbf{t}$ is a Sturmian sequence with $t_0=t_1=1$, then the sequence $(d_n)_{n\geq 1}$ is a 
Sturmian sequence. 
\end{resultD}
\begin{proof}
Note that $\mu_n$ is the number of zeros between the $n$-th and the $(n+1)$-th occurrences of 
`$1$' in $\mathbf{t}$. Write $\mathbf{u}:=(\mu_n)_{n\geq 1}$. Since $\mathbf{t}$ is a Sturmian sequence in which `$00$' does not occur (see above), we have $\mu_n \in \{0,1\}$. Moreover, 
\[\mathbf{t}=\varphi(\mathbf{u})\]
where $\varphi$ is the morphism $0\to 1$ and $1\to 10$. By \cite[Corollary 2.3.3]{M02}, we know that $\mathbf{u}$ is also a Sturmian sequence.

 From Lemma \ref{star-periodic}, we obtain that for all $n\geq 1$,
\[
\alpha_n=\begin{cases}
1, & \text{ if }\mu_n=0,\\
2, & \text{ if }\mu_n=1.
\end{cases}
\]
Then by Eq. \eqref{eq:sma}, for all $n\geq 1$,
\[
d_n=\mu_{n}+\alpha_{n}+1=2(\mu_{n}+1)\in\{2,4\}.
\]
This implies that the difference sequence $(d_{n})_{n\geq 1}$ is the image of $(\mu_{n})_{n\geq 1}$ under the coding $0\to 2$ and $1\to 4$. So $(d_{n})_{n\geq 1}$ is a Sturmian sequence.
\end{proof}

\section{Subword complexity}
\label{six}

We close with a conjecture about the subword complexity of the 
infinite fixed points of the morphisms $\tau_k$.  
The subword complexity is the function
counting the number of distinct factors of length $n$.

\begin{conjecture}
Let $(f_n)_{n \geq 1}$ be the subword complexity of $\tau_k^\infty(1)$,
and define $d_n = f_{n+1} - f_n$ for $n \geq 1$.
Then $d_1 d_2 d_3 \cdots = 1^{a_1} 2^{a_2} 1^{a_3} 2^{a_4} \cdots,$
where
\begin{align*}
a_1 &= k-1; \\
a_2 &= k; \\
a_{2n} &= a_{2n-1} + a_{2n-2}, \quad n \geq 2; \\
a_{2n+1} &= k a_{2n} + k (-1)^n, \quad n \geq 1.
\end{align*}
\end{conjecture}

For example, consider the case $k = 3$.  Then 
\begin{align*}
(f_n)_{n \geq 1} &= (2,3,4,6,8,10,11,12,13,14,15,16,18,20,22,\ldots),\\
(d_n)_{n \geq 1} &= (1,1,2,2,2,1,1,1,1,1,1,2,\ldots),\\
(a_n)_{n \geq 1} &= (2,3,6,9,30,39,114,153,\ldots).
\end{align*}

\section*{Acknowledgements}
The research of Z.-X. Wen was partially supported by NSFC (Nos.~11431007, 11871295). The research of W. Wu was partially supported by China Scholarship Council (File No.~201906155024) and Guangdong Natural Science Foundation (Nos.~2018B0303110005, 2018A030313971).

\end{document}